\documentclass[a4paper,11 pt,reqno]{amsart}
\usepackage{a4,amsbsy,amscd,amssymb,amstext,amsmath,latexsym,oldgerm,enumerate,verbatim,graphicx,xy}

\makeatletter                                       
\def\l@subsection{\@tocline{2}{0pt}{2.5pc}{2.5pc}{}}
\makeatother                                        

\makeatletter                                                  
\def\chapter{\clearpage\thispagestyle{plain}\global\@topnum\z@ 
\@afterindenttrue \secdef\@chapter\@schapter}
\makeatother

\newtheorem{thmgl} {Theorem}    
\newtheorem{propgl}{Proposition}
\newtheorem{lemgl} {Lemma}

\newtheorem{cornn}{Corollary}

\theoremstyle{definition}

\newtheorem{remgl} {Remark}
\newtheorem{remsgl} [remgl]{Remarks}

\newtheorem{remnn}{Remark}

\newcommand{\mf}{\mathfrak}
\newcommand{\mc}{\mathcal}
\newcommand{\mb}{\mathbb}

\newcommand{\nts}{\negthinspace}     
\newcommand{\Nts}{\nts\nts}

\newcommand{\ncd}{\nts\cdot\nts}
\newcommand{\ov}{\overline}

\newcommand{\sm}{\setminus}         

\newcommand{\ot}{\otimes}           
\newcommand{\la}{\langle}
\newcommand{\ra}{\rangle}

\newcommand{\Hom}{{\rm Hom}}        
\newcommand{\End}{{\rm End}}

\newcommand{\Mat}{{\rm Mat}}

\newcommand{\Sym}{{\rm Sym}} 

\newcommand{\sgn}{{\rm sgn}}

\newcommand{\id}{{\rm id}}
\newcommand{\gr}{{\rm gr}}   

\newcommand{\g}{\mf{g}}

\let\ttie\t
\newcommand{\tie}[1]{{\let\t\ttie \ttie#1}}
\renewcommand{\t}{\mf{t}}  

\newcommand{\Lie}{{\rm Lie}}
\newcommand{\GL}{{\rm GL}}

\newcommand{\Ort}{{\rm O}}
\newcommand{\SO}{{\rm SO}}

\newcommand{\Sp}{{\rm Sp}} 

\xyoption{all}

\newcommand{\sleq}{\mbox{\fontsize{2}{3}\selectfont $\leq $}}
\newcommand{\stimes}{\mbox{\fontsize{7}{6}\selectfont $\times $}}
\newcommand{\e}{\epsilon}
\newcommand{\rot}{\rotatebox}
\newcommand{\SpM}{{\rm SpM}}
\newcommand{\GSp}{{\rm GSp}}
\begin{document}

\title{The symplectic ideal and a double centraliser theorem}
\author{Rudolf Tange}

\begin{abstract}
We interpret a result of S.\ Oehms as a statement about the symplectic ideal. We use this result to prove a double centraliser theorem for the symplectic group acting on $\bigoplus_{r=0}^s\otimes^rV$, where $V$ is the natural module for the symplectic group. This result was obtained in characteristic zero by H.\ Weyl. Furthermore we use this to extend to arbitrary connected reductive groups $G$ with simply connected derived group the earlier result of the author that the algebra $K[G]^\g$ of infinitesimal invariants in the algebra of regular functions on $G$ is a unique factorisation domain.
\end{abstract}

\address{School of Mathematics, University of Southampton, Highfield, SO17 1BJ, UK}
\email{rtange@maths.soton.ac.uk} \maketitle

\section*{Introduction}
Throughout this note $k$ denotes an infinite field, $K$ denotes the algebraic closure of $k$, $n$ and $s$ are positive integers, $V$ is the vector space $k^n$ and $\Mat_n=\Mat_n(k)$ is the $k$-algebra of $n\times n$-matrices acting on $V$ as vector space endomorphisms. We denote the vector space $\bigoplus_{r=0}^s\otimes^rV$ by $T^{\sleq s}(V)$. A matrix $u\in\Mat_n$ acts on $\otimes^sV$ by $u(x_1\otimes\cdots\otimes x_s)=u(x_1)\otimes\cdots\otimes u(x_s)$. For a subset $X$ of $\Mat_n$ we denote by $\mc{E}^s(X)$ and $\mc{E}^{\sleq s}(X)$ the {\it enveloping algebra} of $X$ in $\End_k(\otimes^sV)$ and $\End_k(T^{\sleq s}(V))$ respectively, that is, the subalgebra generated by the endomorphisms corresponding to the elements of $X$. Let $A$ be an associative $k$-algebra. The {\it centraliser algebra}\footnote{also called {\it commutating algebra, commutator algebra} or {\it commuting algebra}.} of a subalgebra $B$ of $A$ is defined as the subalgebra of $A$ that consists of the elements of $A$ that commute with all elements of $B$. We say that the {\it double centraliser theorem} holds for a subalgebra $B$ of $A$ if $B$ is equal to the centraliser algebra of its centraliser algebra. We say that the double centraliser theorem holds for a subset $X$ of $\Mat_n$ acting on $\otimes^sV$ or $T^{\sleq s}(V)$ if it holds for the corresponding enveloping algebra.

Now assume that $k=K=\mb{C}$ and let $G$ be a subgroup of $\GL_n(\mb C)$. In \cite{Br} Brauer showed that, if the double centraliser theorem holds for $G$, one can determine defining equations for $\mc{E}^s(G)$ by first determining the centraliser algebra. Determining the centraliser algebra turned out to be equivalent to a version of the first fundamental theorem of invariant theory for $G$.\footnote{In the case $G=\GL_n$ the centraliser algebra was already known and therefore yielded another proof of this version of the first fundamental theorem for $\GL_n$. In characteristic zero one can deduce from this the general first fundamental theorem for $\GL_n$. See \cite{Wey} p.139.} His methods applied to any semisimple complex Lie subgroup $G$ of $\GL_n(\mb C)$ for which there is a first fundamental theorem of invariant theory, in particular to the symplectic group $\Sp_{2m}(\mb C)$. In these cases we have namely that all finite dimensional Lie group representations of $G$ are semisimple and the double centraliser theorem holds for semisimple subalgebras of $\Mat_n(\mb{C})$.

Let $I$ be the ideal of polynomial functions on $\Mat_{2m}$ that vanish on $\Sp_{2m}$. In \cite{Wey} Weyl showed that finding generators $g_1,\ldots,g_r$ of $I$ which have the property that any $f\in I$ can be expressed as $f=\sum_ia_ig_i$ with $\deg(a_ig_i)\leq\deg(f)$ for certain polynomial functions $a_1,\ldots,a_r$ on $\Mat_{2m}$, is equivalent to determining defining equations for $\mc{E}^{\sleq s}(\Sp_{2m})$, for $s$ arbitrary. Then he used Brauers method, which also applies to $T^{\sleq s}(V)$, to determine such generators of $I$.

For the classical groups the first fundamental theorem of invariant theory has been generalised to positive characteristic by De Concini and Procesi in \cite{DeCProc}, but the algebras $\mc{E}^s(G)$ and $\mc{E}^{\sleq s}(G)$, $G$ classical, are no longer semisimple for all $s$ in positive characteristic. In Section~\ref{s.symplecticideal} and \ref{s.doublecentraliser} of this note we will reverse Weyls procedure and show that we can easily derive the double centraliser theorem for $\Sp_{2m}$ acting on $T^{\sleq s}(V)$ once we have generators for the symplectic ideal with the abovementioned property. All this relies on work of Oehms, De Concini and Procesi and Berele.

In Section~\ref{s.liealginvs} we derive another consequence from the results on the symplectic ideal in Section~\ref{s.symplecticideal}. We show that for a connected reductive group $G$ with simply connected derived group, the algebra $K[G]^\g$ of infinitesimal invariants in the algebra of regular functions on $G$ is a unique factorisation domain.

\section{The symplectic ideal and the symplectic enveloping algebra}\label{s.symplecticideal}

We begin by considering the following property of a set of generators $g_1,\ldots,g_r$ of an ideal $I$ in the polynomial ring $k[x_1,\ldots,x_n]$:

{\it Any $f\in I$ can be expressed as

\begin{equation}
f=\sum_{i=1}^ra_ig_i\text{ with }\deg(a_ig_i)\leq\deg(f) \label{eq.idealgensprop}
\end{equation}
for certain $a_1,\ldots,a_r\in k[x_1,\ldots,x_n]$.}

This property is related to the associated graded ideal for which we now introduce some notation. Let $A$ be a commutative algebra over $k$ with a filtration $A_0\subseteq A_1\subseteq A_2\cdots$. The associated graded algebra is denoted by $\gr(A)$. If $x\in A_i\sm A_{i-1}$, then we put $\deg(x)=i$ and $\gr(x)=x+A_{i-1}\in\gr(A)^i=A_i/A_{i-1}$. For an ideal $I$ of $A$ we have $\gr(A/I)\cong \gr(A)/\gr(I)$; see \cite{Bou}, Chapter~3, \S~ 2.4. In particular, when $A$ is graded, we have $\gr(A/I)\cong A/\gr(I)$.

\begin{lemgl}\label{lem.idealgens}
Let $I$ be an ideal of the polynomial ring $k[x_1,\ldots,x_n]$ which is generated by the nonzero elements $g_1,\ldots,g_r$. Then these generators have property \eqref{eq.idealgensprop} if and only if the elements $\gr(g_1),\ldots,\gr(g_r)$ generate the ideal $\gr(I)$.
\end{lemgl}

\begin{proof}
Assume that $g_1,\ldots,g_r$ is set of generators of $I$ which have property \eqref{eq.idealgensprop} and let $\gr(f)$ be a nonzero homogeneous element of $\gr(I)$. Then $f=\sum_{i=1}^ra_ig_i$ with $\deg(a_ig_i)\leq\deg(f)$ for certain $a_1,\ldots,a_r\in k[x_1,\ldots,x_n]$. So $\gr(f)=\sum_i\gr(a_i)\gr(g_i)$ where the sum is over all indices $i$ such that $\deg(a_ig_i)=\deg(f)$. Thus the $\gr(g_i)$ generate $\gr(I)$. Now assume that the latter is the case and let $f$ be a nonzero element of $I$. We can write $\gr(f)=\sum_{i=1}^ra_i\gr(g_i)$ where $a_i$ is homogeneous of degree $\deg(f)-\deg(g_i)$. Then $f-\sum_{i=1}^ra_ig_i$ is of strictly lower degree than $f$ and we can finish by induction.
\end{proof}

The symmetric group $\Sym_r$ on $s$ symbols acts on $\otimes^rV$ via $\pi\cdot(x_1\otimes\cdots\otimes x_r)=x_{(\pi^{-1})_1}\otimes\cdots\otimes x_{(\pi^{-1})_r}$. Denote the group algebra of $\Sym_r$ over $k$ by $k\la\Sym_r\ra$. The elements of $\End_{k\la\Sym_r\ra}(\otimes^rV)$, i.e. the k-linear endomorphisms of $\otimes^rV$ commuting with the action of $\Sym_r$, were called {\it bisymmetric substitutions} by Weyl. Clearly all enveloping algebras ${\mc E}^r(X)$ consist of bisymmetric substitutions.

To state the proposition below correctly we need some auxiliary notation. For a subset $S$ of a fixed finite dimensional $k$-vector space $W$ we denote the $k$-algebra of {\it polynomial functions} on $S$, i.e. functions on $S$ that are restrictions of polynomial functions on $W$, by $k[S]$. The algebra $k[W]$ is a polynomial algebra in the elements of any basis of $W^*$, because $k$ is infinite. In particular, $k[\Mat_n]$ is a polynomial algebra in the matrix entry functions.

Now take $W=\Mat_n$. Clearly $k[S]$ is isomorphic to $k[\Mat_n]/I$, where $I$ is the ideal of polynomial functions on $\Mat_n$ that vanish on $S$. The algebra $k[S]$ inherits a filtration from $k[\Mat_n]$. We denote the filtration subspace of index $s$ by $k[S]^{\sleq s}$. If $S$ is closed under multiplication by nonzero scalars, then the ideal $I$ will be homogeneous and $k[S]$ is a graded algebra. We denote the graded subspace of index $s$ by $k[S]^s$. If $S$ is a submonoid of $\Mat_n$, then $I$ is a biideal and $k[S]$ is a $k$-bialgebra. The subspace $k[S]^{\sleq s}$ is then a subcoalgebra of $k[S]$ and if $S$ contains the nonzero multiples of the identity, then $k[S]^s$ is a subcoalgebra of $k[S]$. We denote the closure of a subset $S$ of $\Mat_n$ under multiplication by nonzero scalars by $k^{\stimes} S$. Note that in this notation $k[\GL_n]$ is equal to $k[\Mat_n]$ and not to $k[\Mat_n][\det^{-1}]$.

After Proposition~\ref{prop.envalg} we will only use this notation in the following situation where it is in accordance with the standard notation. If $X$ is a $k$-defined closed subvariety of $W(K):=K\otimes_kW$ such that the set $S$ of $k$-defined points in $X$ is dense in $X$, then $k[S]$ in the notation here is naturally isomorphic to the algebra $k[X]$ of regular functions on $X$ that are defined over $k$. This applies for example to the symplectic group $\Sp_n$ as we shall see later. The next proposition is essentially due to Schur and Weyl.

\begin{propgl}\label{prop.envalg}
Let $M$ be a submonoid of $\Mat_n$. Then
\begin{enumerate}[(i)]
\item the natural map $\mc{E}^s(M)^*\to k[k^{\stimes} M]^s$ is an isomorphism of coalgebras.
\item the natural map $\mc{E}^{\sleq s}(M)^*\to k[M]^{\sleq s}$ is an isomorphism of coalgebras.
\item $\mc{E}^s(\GL_n)=\End_{k\la\Sym_s\ra}(\otimes^sV)$.
\item $\mc{E}^{\sleq s}(\GL_n)=\End_{k\la \mf{S}\ra}(T^{\sleq s}(V))$, where $\mf{S}=1\times1\times\Sym_2\times\cdots\times\Sym_s$.
\end{enumerate}
\end{propgl}

\begin{proof}
The natural map in (ii) composes $f\in \mc{E}^{\sleq s}(M)^*$ with the monoid homomorphism $M\to\mc{E}^{\sleq s}(M)$. The result is clearly a polynomial function on $M$ of degree $\le s$. That the natural map is injective follows from the fact that endomorphisms of $T^{\sleq s}(V)$ that represent the elements of $M$ span $\mc{E}^{\sleq s}(M)$, since $M$ is a submonoid of $\Mat_n$.

The algebra $\mc{E}^s(M)$ is equal to $\mc{E}^s(k^{\stimes}M)$, since $\alpha\nts\cdot\nts\id\in\Mat_n$ is represented by $\alpha^s\nts\cdot\nts\id$ on $\otimes^sV$. So in the case of (i) we may assume that $M$ contains the nonzero multiples of the identity. The injectivity of the natural map, which composes $f\in \mc{E}^s(M)^*$ with $M\to\mc{E}^s(M)$, now follows as in (ii).

By considering the commutative diagram below and its graded version in case $M$ contains the multiples of the identity, we see that it is now sufficient to prove (i) and (ii) for $M=\Mat_n$, since the two vertical arrows are surjective morphisms of coalgebras.
\begin{equation}
\label{eq.envalg}
\xymatrix{
\mc{E}^{\sleq s}(\Mat_n)^*\ar[r]\ar@{->>}[d]&k[\Mat_n]^{\sleq s}\ar@{->>}[d]\\
\mc{E}^{\sleq s}(M)^*\ar[r]&k[M]^{\sleq s}
}
\end{equation}

We first consider the natural map of (i) for $M=\GL_n$. To avoid confusion, we will write $k[\Mat_n]$ instead of $k[\GL_n]$, since these algebras are canonically isomorphic. It comes from the natural map $\End_k(\otimes^sV)^*\to k[\Mat_n]^s$. Let $I=I(n,s)=\{1,\ldots,n\}^s$ as in \cite[2.1]{Green}. We have a standard basis $(E_{i,j})_{(i,j)\in I\times I}$ of $\End_k(\otimes^sV)$. Our map, maps the element $E_{i,j}^*$ of the dual basis to Green's $c_{i,j}$. The $c_{i,j}$ form a basis when we pass to $(I\times I)/\Nts\sim$, where $\sim$ is the equivalence relation given by the action of the $\Sym_s$ on $I\times I$ defined in \cite[2.1]{Green}. Now we form the dual basis $c_{\ov{(i,j)}}^*$ and define $c_{i,j}^*=c_{\ov{(i,j)}}^*$\,. Here $\ov{(i,j)}$ denotes the canonical image of $(i,j)$ in $(I\times I)/\Nts\sim$\,. Then the transpose of our map maps $c_{i,j}^*$ to $\sum_{(k,l)\sim(i,j)}E_{k,l}$, which is precisely what Green's map $(k[\Mat_n]^s)^*\to\End_k(\otimes^sV)$ does. Assertion (i) for $\GL_n$ and (iii) now follow from \cite[2.6]{Green}. Note that (iii) shows that $\mc{E}^s(\Mat_n)=\mc{E}^s(\GL_n)$.

Now we consider the natural map of (ii) for $M=\GL_n$. We have seen that it is injective. It is also surjective, since for any polynomial function $f$ on $\GL_n$ of degree $\le s$ we can find an element of $\End(T^{\sleq s}(V))^*\cong T^{\sleq s}(\Mat_n^*)$ that is mapped to $f$. We have $\mc{E}^{\sleq s}(\GL_n)\subseteq\bigoplus_{r=0}^s\mc{E}^r(\GL_n)$. Since the natural maps of (i) and (ii) for $M=\GL_n$ are bijective we must have equality. It is now also clear that (ii) is an isomorphism of coalgebras.

Since the image of $k\la \mf{S}\ra$ in $\End_k(T^{\sleq s}(V))$ contains the projections given by the direct sum decomposition of $T^{\sleq s}(V)$, we have that $\End_{k\la \mf{S}\ra}(T^{\sleq s}(V))$ is equal to $\bigoplus_{r=0}^s\End_{k\la\Sym_r\ra}(\otimes^rV)$\,: it consists of $(s+1)$-tuples of "bisymmetric substitutions". So (iv) now follows from (iii) and it is now also clear that $\mc{E}^{\sleq s}(\GL_n)=\mc{E}^{\sleq s}(\Mat_n)$.
\end{proof}

\begin{remsgl}\label{rems.envalg}
1.\ Assertion (iv) of Proposition~\ref{prop.envalg} was proved in a different way by Weyl, see \cite[Thm.~4.4.E]{Wey}.\\
2.\ Proposition~\ref{prop.envalg} shows that $\mc{E}^{\sleq s}(M)$ only depends on the ideal $I$ of polynomial functions that vanish on $M$ and that $\mc{E}^s(M)$ only depends on the biggest homogeneous ideal contained in $I$, i.e. the ideal generated by the homogeneous polynomial functions on $\Mat_n$ that vanish on $M$.\\
3.\ Assume that $k=K$ is algebraically closed. Let $G$ be a closed subgroup of $\GL_n$ and let $A$ be the algebra of polynomial functions on $G$. Then the representations of the algebra $\mc{E}^{\sleq s}(G)= (A^{\sleq s})^*$ are precisely the rational representations of $G$ whose coefficients are polynomial functions on $G$ that are of filtration degree $\leq s$. If $G$ contains the nonzero multiples of the identity, then the representations of the algebra $\mc{E}^s(G)= (A^s)^*$ are precisely the rational representations of $G$ whose coefficients are homogeneous polynomial functions on $G$ that are of degree $s$. In case $G=\GL_n$, $(A^s)^*$ is the well-known {\it Schur algebra} denoted by $S_K(n,s)$ in \cite{Green}. Similar remarks apply to a closed submonoid $M$ of $\Mat_n$. Then the coefficients of a rational representation of $M$ are automatically polynomial, since all regular functions on $M$ are polynomial.
\end{remsgl}

We now come to the problem of finding defining equations for the enveloping algebras $\mc{E}^s(M)$ and $\mc{E}^{\sleq s}(M)$. We are only interested in defining equations within $\mc{E}^s(\Mat_n))$ and $\mc{E}^{\sleq s}(\Mat_n)$. By Proposition~\ref{prop.envalg}(iii) and (iv) one can then obtain a complete set of defining equations (within $\End_k(\otimes^sV)$ and $\oplus_{r=0}^s\End_k(\otimes^rV)$) by adding the equations of "bisymmetry". See \cite[(30)]{Br} or \cite[III (6.8)]{Wey}. Statement (i) of the next corollary is a formalisation of an idea of Weyl; see \cite[p 142]{Wey}.

\begin{cornn}
Let $M$ be a submonoid of $\Mat_n$, let $I$ be the ideal of polynomial functions on $\Mat_n$ that vanish on $M$ and let $I_{hom}$ be the biggest homogeneous ideal contained in $I$. Furthermore, let $g_1,\ldots,g_k$ be nonzero elements of $I$ and let $h_1,\ldots,h_l$ be nonzero homogeneous elements of $I_{hom}$. Denote the isomorphism $k[\Mat_n]^{\sleq s}\to \mc{E}^{\sleq s}(\Mat_n)^*$ by $\eta$ and the isomorphism $k[\Mat_n]^s\to \mc{E}^s(\Mat_n)^*$ by $\theta$. Then
\begin{enumerate}[(i)]
\item The elements $g_1,\ldots,g_k$ are generators of $I$ with property \eqref{eq.idealgensprop} if and only if for each $s\ge 1$ the functionals $\eta(g_im_i)$, where the $m_i$ are arbitrary monomials in the matrix entries of degree $\le s-\deg(g_i)$, define the algebra $\mc{E}^{\sleq s}(M)$.
\item The elements $h_1,\ldots,h_k$ are generators of $I_{hom}$ if and only if for each $s\ge 1$ the functionals $\theta(h_im_i)$, where the $m_i$ are arbitrary monomials in the matrix entries of degree $s-\deg(h_i)$, define the algebra $\mc{E}^s(M)$.
\end{enumerate}
\end{cornn}

\begin{proof}
Note that $I$ is  a proper ideal, since $M$ is nonempty. Denote the subspace of $I$ that consists of the elements of $I$ of degree $\le s$ by $I^{\sleq s}$. Then the $g_i$ are generators of $I$ with property \eqref{eq.idealgensprop} if and only if for each $s\ge 1$ the polynomials $g_im_i$, where the $m_i$ are arbitrary monomials in the matrix entries of degree $\le s-\deg(g_i)$, span $I^{\sleq s}$. Assertion (i) now follows from the fact that, because of the commutativity of \eqref{eq.envalg}, $\eta$ maps $I^{\sleq s}$ to the space of linear functionals on $\mc{E}^{\sleq s}(\Mat_n)$ that vanish on $\mc{E}^{\sleq s}(M)$. Assertion (ii) is proved in a similar way.
\end{proof}

\begin{remnn}
In \cite[Thm.'s 9.5a) and 10.5a)]{Doty} Doty determines generators for the defining ideals of the symplectic and orthogonal monoids in characteristic zero by a method which is essentially assertion (ii) of the above corollary. In characteristic zero one can namely use the defining equations of the symplectic and orthogonal enveloping algebras in $\End_k(\otimes^sV)$ as obtained by Brauer in \cite[(46),(47)]{Br}.
\end{remnn}

In the remainder of this note we assume that $n=2m$ is an even integer. We first intoduce some notation for the symplectic group which closely follows that of \cite{Oe}. Let $i\mapsto i'$ be the involution  of $\{1,\ldots,n\}$ defined by $i':=n+1-i$. Set $\e_i=1$ if $i\le m$ and $\e_i=-1$ if $i>m$ and define the $n\times n$-matrix $J$ with coefficients in $k$ by $J_{ij}=\delta_{ij'}\e_i$. So

$$J=
\begin{bmatrix}
&&&&&1\\
&0&&&\rot{72}{$\ddots$}&\\
&&&1&&\\
&&-1&&&\\
&\rot{72}{$\ddots$}&&&0&\\
-1&&&&&
\end{bmatrix}.
$$

On $V$ we define the nondegenerate symplectic form $\la\ ,\ \ra$ by
$$\la u,v\ra:=u^TJv=\sum_{i=1}^n\e_iu_iv_{i'}\ .$$
The {\it symplectic group} $\Sp_n=\Sp_n(k)$ is defined as the set of $n\times n$-matrices over $k$ that satisfy $A^TJA=J$, i.e. the matrices for which the corresponding endomorphism of $V$ preserves the form $\la\ ,\ \ra$. Clearly those matrices are invertible and $\Sp_n$ is a subgroup of $\GL_n$. Furthermore $\Sp_n(K)$ is an algebraic group which is connected, semisimple and defined and split over the prime field. This implies that all root subgroups $U_\alpha(K)$ with respect to some $k$-split maximal torus are defined over $k$. Since clearly $U_\alpha(k)$ is dense in $U_\alpha(K)$, we must have that $\Sp_n$ is dense in $\Sp_n(K)$.

Note that $A^TJA$ consists of the scalar products of the columns of $A$ with each other and that $AJA^T$ consists of the scalar products of the rows of $A$ with each other. An easy calculation shows that the condition $A^TJA=J$ is equivalent to the condition $AJA^T=J$. We denote the $(i,j)^{\rm th}$ entry function on $\Mat_n$ by $x_{ij}$. Define
\begin{equation}
g_{ij}:=\sum_{l=1}^n\e_lx_{li}x_{l'j}\text{\quad and\quad}\ov{g}_{ij}:=\sum_{l=1}^n\e_lx_{il}x_{jl'}.
\end{equation}
The condition $A^TJA=J$ means that we require the functions $g_{ij}-\delta_{ij'}\e_i$, $i<j$, to vanish on $A$. The condition $AJA^T=J$ means that we require the functions $\ov{g}_{ij}-\delta_{ij'}\e_i$, $i<j$, to vanish on $A$.  We call the ideal of polynomial functions on $\Mat_n$ that vanish on $\Sp_n$ the {\it symplectic ideal}. It is well known that both sets of functions separately generate the symplectic ideal in case $k$ is algebraically closed. In view of the density of $\Sp_n$ in $\Sp_n(K)$ this must then also hold for an arbitrary infinite field $k$.

In the sequel we will also need the {\it symplectic monoid} $\SpM_n$ and the {\it symplectic similitude group} $\GSp_n$ as introduced in \cite{Doty}; see also \cite{Oe} and \cite{Dotyetal}. The symplectic monoid $\SpM_n$ is defined as the set of matrices $A$ for which there exists a scalar $d(A)\in k$ such that $A^TJA=AJA^T=d(A)J$. Note that if $d(A)\ne0$, $A^TJA=d(A)J$ is equivalent to $AJA^T=d(A)J$. Clearly $\SpM_n$ is a submonoid of $\Mat_n$, in fact it is the set of $k$-defined points of $\SpM_n(K)$ which is a $k$-defined closed submonoid of $\Mat_n(K)$. This follows from \cite[Cor.~6.2]{Oe}. The function $d$ is a polynomial function on $\SpM_n$, it is called the {\it coefficient of dilation}. We have
for all $i\in\{1,\ldots,n\}$
$$d=\e_ig_{ii'}=\e_i\ov{g}_{ii'}\text{ \ on }\SpM_n.$$
The {\it symplectic similitude group} $\GSp_n$ is defined as the group of invertible elements in $\SpM_n$. We have $\GSp_n(K)=K^\times\Sp_n(K)$ and $\GSp_n$ is Zariski dense in $\GSp_n(K)$ and $\SpM_n(K)$. So a polynomial function on $\Mat_n$ vanishes on $\GSp_n$ if and only if it vanishes on $\SpM_n$ and the ideal of $k[\Mat_n]$ that consists of these functions is the biggest homogeneous ideal contained in the symplectic ideal.

We will now give defining equations for ${\mc E}^{\sleq s}(\Sp_n)$. As stated before the Corollary to Proposition~\ref{prop.envalg}, we are only interested in defining equations within the algebra ${\mc E}^{\sleq s}(\Mat_n)$. This algebra consists of $(s+1)$-tuples $(A_0,\ldots,A_s)$ of endomorphisms with $A_r\in\End_k(\otimes^rV)$ bisymmetric for all $r\in\{0,\ldots,s\}$. The standard basis $(e_1,\ldots,e_n)$ of $V=k^n$ gives bases for the vector spaces $\otimes^rV$. The entry of index $((i_1,\ldots,i_r),(j_1,\ldots,j_r))$ of the matrix of $A_r$ with respect to these bases is denoted by $a_{i_1\cdots i_r,j_1\cdots j_r}$.

The first statement in (i) of the theorem below was already remarked by S.~Oehms \cite[p 38]{Oe}.

\begin{thmgl}\label{thm.symplecticideal}
The following holds
\begin{enumerate}[(i)]
\item $\gr(I)$ is generated by the elements $g_{ij}$ and $\ov{g}_{ij}$, $i<j$. Furthermore, $\gr(k[\Sp_n])$ is isomorphic to the algebra of polynomial functions on the set of $n\times n$-matrices over $k$ whose row and column space are totally singular and this algebra is an integral domain.
\item The sets of elements $g_{ij}-\delta_{ij'}\e_i$ and $\ov{g}_{ij}-\delta_{ij'}\e_i$, $i<j$, form together a set of generators of the symplectic ideal with property \eqref{eq.idealgensprop}.
\item The symplectic enveloping algebra in $\End_k(T^{\sleq s}(V))$ is defined by the following equations for $r=2,\ldots,s$
\begin{align}
\sum_{l=1}^n\e_la_{ll'i_3\cdots i_r,j_1\cdots j_r}=
\delta_{j_1,j_2'}\e_{j_1}a_{i_3\cdots i_r,j_3\cdots j_r}\,,\label{eq.envalgeqs1}\\
\sum_{l=1}^n\e_la_{i_1\cdots i_r,ll'j_3\cdots j_r}=
\delta_{i_1,i_2'}\e_{i_1}a_{i_3\cdots i_r,j_3\cdots j_r}\,.\label{eq.envalgeqs2}
\end{align}
\end{enumerate}
\end{thmgl}

\begin{proof}
(i).\ Let $J$ be the ideal of $k[\Mat_n]$ generated by the elements  $g_{ij}$ and $\ov{g}_{ij}$, $i<j$. Then the zero set of $J$ is the set of $n\times n$-matrices over $k$ whose row and column space are totally singular. It is also the zero set of $d$ in $\SpM_n$. In fact it follows from \cite[Cor.~6.2]{Oe} that $k[\SpM_n]/(d)\cong k[\Mat_n]/J$. This result also implies that $k[\SpM_n]/(d-1)\cong k[\Mat_n]/I=k[\Sp_n]$. Since $d$ is nonzero and homogeneous of degree $2>0$, we have $\gr(d-1)=d$ and $\gr((d-1))=(d)$. Therefore
$$k[\Mat_n]/J\cong k[\SpM_n]/(d)\cong\gr(k[\Sp_n])\cong k[\Mat_n]/\gr(I).$$
Since the composite is the epimorphism given by the inclusion $J\subseteq\gr(I)$, we must have $J=\gr(I)$.

Let $\mc V$ be the variety of $n\times n$-matrices over $K$ whose row and column space are totally singular. Then we have an epimorphism $K[\SpM_n]/(d)\twoheadrightarrow K[\mc V]$ of graded $K$-algebras. To show that it is an isomorphism it suffices to show that for each $l$ the dimensions of the graded pieces of degree $l$ of both algebras are equal. By \cite[Thm.~6.1]{Oe} the $l^{\rm th}$ graded piece of $K[\SpM_n]/(d)$ has a basis consisting of all bideterminants $(T_1|\,T_2)$, where $T_1$ and $T_2$ are symplectic standard tableaux in the sense of King of the same shape $\lambda$ which is a partition of $l$ with at most $m$ parts (notation: $\lambda\vdash_m l$). So its dimension is $\sum_{\lambda\vdash_m l}N_\lambda^2$, where $N_\lambda$ is the number of symplectic standard tableaux in the sense of King of shape $\lambda$. By \cite[Thm.~6.1]{DeC} the $l^{\rm th}$ graded piece of $K[\mc V]$ has a basis consisting of all bideterminants $(T_1|\,T_2)$, where $T_1$ and $T_2$ are symplectic standard tableaux in the sense of De Concini of the same shape $\lambda\vdash_m l$.\footnote{De Concini uses different terminology, but it is clear that his condition of symplectic standardness is imposed on $T_1$ and $T_2$ separately.} So its dimension is $\sum_{\lambda\vdash_m l}{N'_\lambda}^2$, where $N'_\lambda$ is the number of symplectic standard tableaux in the sense of De Concini of shape $\lambda$. Let $Y(\lambda)$ be the irreducible $\Sp_n(\mb C)$-module associated to $\lambda\vdash_m l$. Then $N_\lambda=\dim(Y(\lambda))=N'_\lambda$, by \cite{Ber} or \cite[Thm.~2.3b]{Don} and \cite[Thm.'s 4.8, 4.10, Prop. 4.10]{DeC}. Note that we now also know that $\mc V$ is defined over $k$.

We will now show that $\mc V$ is irreducible. Let ${\mc V}_0$ be the set of $n\times n$-matrices whose row and column space lie in the $K$-span $V_0$ of $e_1,\ldots,e_m$. Note that $V_0$ is a maximal totally singular subspace of $V(K)=K^n$ and that $\mc V_0$ is a $k$-defined vector subspace of $\Mat_n(K)$, where $\Mat_n(K)$ has the standard $k$-structure. Now consider the morphism $\Sp_n(K)\times\mc V_0\times\Sp_n(K)\to\Mat_n(K)$ given by $(A,S,B)\mapsto ASB$. Clearly its image lies in $\mc V$ and by Witt's Lemma its image is equal to $\mc V$. Since $\Sp_n(K)$ and $\mc V_0$ are irreducible we must have that $\mc V$ is irreducible. Furthermore the set of $k$-defined points is dense in $\mc V$, since this holds for $\Sp_n(K)$ and $\mc V_0$ and our morphism is defined over $k$. Note that, since $K[\mc V]$ is an integral domain, $d$ generates a prime ideal in $K[\SpM_n]$ and therefore also in $k[\SpM_n]$. \\
(ii).\ This follows from (i) and Lemma~\ref{lem.idealgens}.\\
(iii).\ This follows immediately from (ii) and (i) of the Corollary to Proposition~\ref{prop.envalg}.
\end{proof}

\begin{remsgl}\label{rems.defeqs}
1.\ The equations in (iii) are the same as those obtained by Weyl \cite[p 174]{Wey} in characteristic zero.\\
2.\ The equations for ${\mc E}^s(\Sp_n)$ that Brauer \cite[(47)]{Br} found in characteristic zero also define this algebra in positive characteristic. This can be shown as follows. Let $I$ be the symplectic ideal. Then \cite[Cor.~6.2]{Oe} gives homogeneous generators of $I_{hom}$.\footnote{In \cite[(14)]{Oe}, $f_{ll'}-f_{kk'}$ should be replaced by $\e_kf_{ll'}-\e_lf_{kk'}$ (or by $\e_lf_{ll'}-\e_kf_{kk'}$).} Here one can avoid the FRT-construction by simply taking \cite[(13)]{Oe} as the definition of $A^s_R(n)$. The stability under base change is then trivial and the proofs of \cite[Thm~6.1 and Cor.~6.2]{Oe} still apply. Now (ii) of the Corollary to Proposition~\ref{prop.envalg} gives the following equations for ${\mc E}^s(\Sp_n)$ within ${\mc E}^s(\Mat_n)$:
$$\delta_{i_1,i_2'}\e_{i_1}\sum_{l=1}^n\e_la_{ll'i_3\cdots i_r,j_1\cdots j_r}=
\delta_{j_1,j_2'}\e_{j_1}\sum_{l=1}^n\e_la_{i_1\cdots i_r,ll'j_3\cdots j_r}.$$
%
%
\end{remsgl}

\section{Infinitesimal invariants in $K[G]$}\label{s.liealginvs}
In this section we assume that $k=K$ is algebraically closed. Furthermore, $G$ is a connected reductive algebraic group over $K$ and $\g=\Lie(G)$ is its Lie algebra. The conjugation action of $G$ on itself induces an action of $G$ on $K[G]$, the algebra of regular functions on $G$. We will refer to this action and its derived $\g$-action as {\it conjugation} actions. The conjugation action of $G$ on $K[G]$ is by algebra automorphisms, so the conjugation action of $\g$ on $K[G]$ is by derivations. This implies that the space $K[G]^\g=\{f\in K[G]\,|\,x\cdot f=0\text{ for all } x\in\g\}$ of {\it infinitesimal invariants} is a subalgebra of $K[G]$. Note that $K[G]^G\subseteq K[G]^\g$ and that $K[G]^p\subseteq K[G]^\g$ if $K$ is of characteristic $p>0$.


\begin{thmgl}\label{thm.ufd}
Assume that $K$ is of characteristic $p>0$ and that the derived group of $G$ is simply connected. Then the invariant algebra $K[G]^\g$ is a unique factorisation domain. Its irreducible elements are the irreducible elements of $K[G]$ that are invariant under $\g$ and the $p$-th powers of the irreducible elements of $K[G]$ that are not invariant under $\g$.
\end{thmgl}

\begin{proof}
By Remark~3 in \cite{T} we have to prove Proposition~1 in that paper for $\Sp_n$. In fact we may assume that $n\ge2$ and that $K$ is of characteristic $2$, but we will not use this. The statement of Proposition~1 in \cite{T} is:

{\it Let $f\in K[G]$ be a  regular function. If the ideal $K[G]f$ is stable under the conjugation action of $\g$ on $K[G]$, then $f$ is $\g$-invariant.}

By Proposition~2 and Lemma~2 in \cite{T} it is sufficient to show that, for the filtration that $K[\Sp_n]$ inherits from $K[\Mat_n]$, the associated graded algebra is an integral domain.
The point is that one can then use a (filtration) degree argument. The result now follows immediately from Theorem~\ref{thm.symplecticideal}(ii).
\end{proof}

\begin{propgl}
The algebras $K[\Sp_n]$, $K[\GSp_n]$ and $K[\SpM_n]$ are unique factorisation domains.
\end{propgl}

\begin{proof}
Since the derived group $D\GSp_n $ of $\GSp_n$ equals $\Sp_n$ which is simply connected, we have by \cite[Cor. to Thm.~1]{T} that $K[\Sp_n]$ and $K[\GSp_n]$ are UFD's. We have $K[\SpM_n][d^{-1}]=K[\GSp_n]$ which is a UFD. Furthermore $d$ generates a prime ideal in $K[\SpM_n]$ by the proof of Theorem~\ref{thm.symplecticideal}(i). So $K[\SpM_n]$ is a UFD by Nagata's Lemma; see e.g. \cite[Lemma 19.20]{Eis}.
\end{proof}

\section{A double centraliser theorem for the symplectic group}\label{s.doublecentraliser}
We begin by defining a version of the symplectic Brauer algebra. For each integer $r\le s$ we fix $r$ vector symbols $x_1,\ldots,x_r$ and $r$ covector symbols $y_1,\ldots,y_r$.\footnote{elements of $V$ are called {\it vectors} and elements of $V^*$ are called {\it covectors}.} To ease notation we did not give these symbols the extra index $r$ that indicates to which integer $r\le s$ they are associated. In fact one may, as the notation suggests, take the $r$-symbols as the first $r$ of the $s$-symbols. Let $u$ and $v$ be nonnegative integers $\le s$ such that $u\equiv v\ ({\rm mod}\ 2)$, that is, such that $u+v$ is even. A {\it $(u,v)$-diagram} is a matching of the $u+v$ symbols $y_1,\ldots,y_u,x_1,\ldots,x_v$ in pairs. Such a diagram is depicted as a graph whose vertices are arranged in two rows. The top row consists of $u$ vertices representing, from left to right, the $u$ covector symbols $y_1,\ldots,y_u$ and the bottom row consists of $v$ vertices representing the $v$ vector symbols $x_1,\ldots,x_v$. An empty row is represented by $\emptyset$. Two vertices are joined if the corresponding symbols are matched. So for example

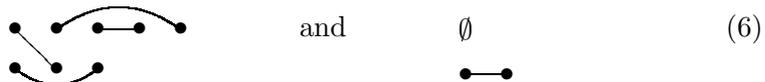
\begin{equation}\label{eq.example}
\xymatrix @R=14pt @C=14pt @M=-2pt{
{\bullet}\ar@{-}[1,1]&{\bullet}\ar@{-}@/^.7pc/[0,3]&{\bullet}{\ar@{-}[0,1]}&{\bullet}&{\bullet}\\
{\bullet}{\ar@{-}@/_.5pc/[0,2]}&{\bullet}&{\bullet}
}
\text{\qquad\qquad and\qquad\qquad}
\xymatrix @R=14pt @C=14pt @M=-2pt{
\emptyset\\
{\bullet}{\ar@{-}[0,1]}&{\bullet}
}
\end{equation}

\vspace{.5cm}

\noindent are a $(5,3)$ and a $(0,2)$-diagram.

If an edge $e$ in a $(u,v)$-diagram is horizontal, then its left endpoint is called its {\it initial point} and its right endpoint is called its {\it terminal point}. If $e$ is not horizontal, then its endpoint in the top row is called its initial point and its endpoint in the bottom row is called its terminal point.

Let $t\in k$. We now define the $k$-algebra ${\mf B}_{\sleq s}(t)$. It has all $(u,v)$-diagrams, $0\le u,v\le s$, $u\equiv v\ ({\rm mod}\ 2)$, as a $k$-basis and its dimension is $\sum_{u=1}^s\sum_{v=1}^sN_{uv}$, where

$$N_{uv}=\begin{cases}
(u+v-1)(u+v-3)\cdots3\ncd1=\frac{(u+v)!}{2^{(u+v)/2}((u+v)/2)!}&\text{if $u\equiv v\ ({\rm mod}\ 2)$,}\\
0&\text{otherwise.}
\end{cases}$$

The multiplication of diagrams is defined as follows. Let $D_{uv}$ be a $(u,v)$-diagram and let $D_{u'v'}$ be a $(u',v')$-diagram. Then we define

$$D_{uv}D_{u'v'}=\begin{cases}
\sgn(D_{uv},D_{u'v'})t^{\gamma(U)}D_{uv'}&\text{if $v=u'$,}\\
0&\text{if $v\ne u'$.}
\end{cases}$$

Here, as in \cite{Br} and \cite{HanWal}, $U$ is the graph which is obtained by putting $D_{uv}$ on top of $D_{u'v'}$ and identifying the vertices of the bottom row of $D_{uv}$ with those of the top row of $D_{u'v'}$, $\gamma(U)$ is the number of cycles in $U$ and $D_{uv'}$ is the $(u,v')$-diagram obtained from $U$ by taking the vertices from top and bottom row in $U$ and matching two vertices if there is a path between them in $U$.

Our definition of the sign $\sgn(D_{uv},D_{u'v'})$ of $D_{uv}$ over $D_{u'v'}$ does not agree with \cite[p 411]{HanWal}. The graph $U$ consists of $(u+v')/2$ paths and $\gamma(U)$ cycles that are of even length and consist entirely of edges in the middle row of $U$. A path in $U$ either consists of one horizontal edge in the top or bottom row of $U$ or it has precisely two non-horizontal edges and all its horizontal edges are in the middle row of $U$. A path of even length has one endpoint in the top row of $U$ and the other endpoint in the bottom row of $U$. A path of odd length has its endpoints both in the top row of $U$ or both in the bottom row of $U$.

Every path $P$ has a "natural" orientation. If $P$ has even length, then the endpoint in the top row is the initial point and the endpoint in the bottom row is the terminal point. If $P$ has odd length, then the leftmost endpoint of $P$ is the initial point and the other endpoint is the terminal point. For a path $P$ or an oriented cycle $P$ we denote the number of edges of $P$ in the middle row of $U$ that are traversed from left to right by $p_{lr}$ and the number of edges of $P$ in the middle row of $U$ that are traversed from right to left by $p_{rl}$. For a cycle or a path $P$ we define
$$\sgn(P)=
\begin{cases}
(-1)^{|p_{lr}-p_{rl}|/2}&\text{ if $p_{lr}+p_{rl}$ is even,}\\
(-1)^{|p_{lr}-p_{rl}-1|/2}&\text{ if $p_{lr}+p_{rl}$ is odd.}
\end{cases}$$
Note that for a cycle this does not depend on the orientation. We now define $\sgn(D_{uv},D_{u'v'})$ as the product of the signs of all paths and cycles in $U$.

Below we will define a natural representation of ${\mf B}_{\sleq s}(n)$ which is the motivation for the definition of the multiplication. We will see later that the multiplication of ${\mf B}_{\sleq s}(t)$ is associative.

With each $(u,v)$-diagram $D$ we can associate a $(u+v)$-multilinear function $F(D)$ on $\bigoplus^uV^*\oplus\bigoplus^vV$ as follows. First we observe that the form $\la\ ,\ \ra$ defines an isomorphism $V\cong V^*$, unique up to sign, and therefore a unique symplectic form on $V^*$ which we also denote by $\la\ ,\ \ra$. Furthermore we put $\la y,x\ra=y(x)=-\la x,y\ra$ for $y\in V^*$ and $x\in V$. If $e$ is an edge in $D$ whose initial point has label $z_1$ and whose terminal point has label $z_2$, then we put $\la e\ra=\la z_1,z_2\ra$. Now we define
\begin{equation}\label{eq.F(D)}
F(D)=\prod_{e\in D}\la e\ra.
\end{equation}
For example, for the first diagram $D$ in \eqref{eq.example} we have

$$F(D)=\la y_1,x_2\ra\la y_2,y_5\ra\la y_3,y_4\ra\la x_1,x_3\ra.$$

Using the $k$-vector space isomorphisms

\begin{equation}\label{eq.moduleiso1}
\End_k(T^{\sleq s}(V))\cong\bigoplus_{0\le u,v\le s}\Hom_k(\otimes^vV,\otimes^uV)
\end{equation} and
\begin{equation}\label{eq.moduleiso2}
\Hom_k(\otimes^vV,\otimes^uV)\cong(\otimes^uV)\otimes(\otimes^vV)^*\cong
\big((\otimes^uV^*)\otimes(\otimes^vV)\big)^*\,,
\end{equation}
we can now associate to each $(u,v)$-diagram $D$ an endomorphism $E(D)$ of $T^{\sleq s}(V)$ as follows. We take $E(D)$ to be the endomorphism of $T^{\sleq s}(V)$ corresponding to $F(D)\in\big((\otimes^uV^*)\otimes(\otimes^vV)\big)^*$. Let $(e_1,\ldots,e_n)$ be the standard basis of $V=k^n$. Then we have a dual basis $(e^*_1,\ldots,e^*_n)$ and also bases of the vector spaces $\otimes^rV$. The matrix $M(D)$ of $E(D)\in\Hom_k(\otimes^vV,\otimes^uV)$ with respect to these bases is then given by
\begin{equation}\label{eq.M(D)}
M(D)_{(i_1,\ldots,i_u),(j_1,\ldots,j_v)}=F(D)(e^*_{i_1},\ldots,e^*_{i_u},e_{j_1},\ldots,e_{j_v}).
\end{equation}
The linear map $E:{\mf B}_{\sleq s}(n)\to\End_k(T^{\sleq s}(V))$ given by the assignment $D\mapsto E(D)$ is a homomorphism of algebras. We indicate a proof.

The isomorphism $\varphi:V^*\to V$ with $\la\varphi(y),x\ra=y(x)$ for all $x\in V$ and $y\in V^*$ maps $e^*_i$ to $\e_{i'}e_{i'}$. So we have $$\la e^*_i,e^*_j\ra=\la e_i,e_j\ra=\e_i\delta_{ij'}.$$
Now one easily checks that

$$
\sum_{l_2,\ldots, l_p=1}^n\la e_{l_1},e_{l_2}\ra\la e_{l_2},e_{l_3}\ra\cdots\la e_{l_p},e_{l_{p+1}}\ra=
\begin{cases}
(-1)^{p/2}\la e^*_{l_1},e_{l_{p+1}}\ra\text{\quad\ \ \,if $p$ is even,}\\
(-1)^{(p-1)/2}\la e_{l_1},e_{l_{p+1}}\ra\text{ if $p$ is odd.}
\end{cases}$$
and that for $p$ even
$$
\sum_{l_1,\ldots, l_p=1}^n\la e_{l_1},e_{l_2}\ra\la e_{l_2},e_{l_3}\ra\cdots
\la e_{l_{p-1}},e_{l_p}\ra\la e_{l_p},e_{l_1}\ra=
(-1)^{p/2}\cdot n\,.$$
If we evaluate an entry of the product $M(D_{uv})M(D_{u'v'})$, $v=u'$, then this entry will be a sum over $v$ indices corresponding to the vertices in the middle row of $U$. Here we use \eqref{eq.F(D)} and \eqref{eq.M(D)}. Now we can distribute these $v$ sums over the paths and cycles in $U$. Each path or cycle gets the sums corresponding to the vertices of the middle row of $U$ that it contains. Then we obtain a product of sums of the above two types where we have to swap the arguments of certain scalar products $\la e_{l_{i+1}},e_{l_{i+2}}\ra$ according to position of the corresponding vertices in the middle row of $U$. Here an extra sign $(-1)^{p_{lr}}$ comes in. Now it follows that the entry of $M(D_{uv})M(D_{u'v'})$ is equal to the corresponding entry of $M(D_{uv}D_{u'v'})$.

We now define some special elements of ${\mf B}_{\sleq s}(t)$.
\begin{equation}
c_r=\hspace{.1cm}
\begin{xy}
(-6.4,2.8)*=<47pt,28pt>{
\xymatrix @R=14pt @C=14pt @M=-2pt{
{\bullet}\ar@{-}[1,0]&\cdots&{\bullet}\ar@{-}[1,0]\\
{\bullet}&\cdots&{\bullet}
}
}*\frm{_\}};%
(-6.4,-5.2)*{\text{$r$ vertices}}
\end{xy}
\text{\hspace{.2cm},\hspace{.3cm}}
b_r=\hspace{1.1cm}
\begin{xy}
(-17.5,2.5)*=<47pt,26pt>{
\xymatrix @R=14pt @C=14pt @M=-2pt{
&&{\bullet}\ar@{-}[1,0]&\cdots&{\bullet}\ar@{-}[1,0]\\
{\bullet}\ar@{-}[0,1]&{\bullet}&{\bullet}&\cdots&{\bullet}
}
}*\frm{^\}};%
(-17.5,11)*{\text{$r$ vertices}}
\end{xy}
\text{\hspace{.4cm}and\hspace{.3cm}}
\ov{b}_r=\hspace{1.1cm}
\begin{xy}
(-17.4,2.8)*=<47pt,28pt>{
\xymatrix @R=14pt @C=14pt @M=-2pt{
{\bullet}\ar@{-}[0,1]&{\bullet}&{\bullet}\ar@{-}[1,0]&\cdots&{\bullet}\ar@{-}[1,0]\\
&&{\bullet}&\cdots&{\bullet}
}
}*\frm{_\}};%
(-17.5,-5.2)*{\text{$r$ vertices}}
\end{xy}
\text{\hspace{.2cm}.}
\end{equation}\vspace{.3cm}

Here $c_r$ is defined for $0\le r\le s$ and $b_r$ and $\ov{b}_r$ are defined for $0\le r\le s-2$.
The unit element of ${\mf B}_{\sleq s}(t)$ is $\sum_{r=0}^sc_r$. The subspace $c_r{\mf B}_{\sleq s}(t)c_r$ is the span of all $(r,r)$-diagrams and is closed under multiplication. It has $c_r$ as a unit element and as an algebra it is isomorphic to the symplectic Brauer algebra ${\mf B}_r(t)$ which was introduced for $t=n$ in \cite{Br}. We identify ${\mf B}_r(t)$ with this subspace of ${\mf B}_{\sleq s}(t)$. Then $\bigoplus_{r=0}^s{\mf B}_r(t)$ is identified with a subalgebra with unit of ${\mf B}_{\sleq s}(t)$. We obtain a natural embedding $k\la\Sym_r\ra\subseteq {\mf B}_r(t)$ by assigning to each $\pi\in\Sym_r$ the $(r,r)$-diagram in which $x_i$ is matched with $y_{\pi_i}$. The action of $\Sym_r$ on $\otimes^rV$ that it inherits from ${\mf B}_r(n)$ is the same as the action mentioned in Section~\ref{s.symplecticideal}.

The algebra ${\mf B}_{\sleq s}(t)$ is generated by the elements $b_r$, $\ov{b}_r$, $0\le r\le s-2$, $c_0$, $c_1$ and the elements of each ${\mf B}_r(t)$, $2\le r\le s$, that correspond to the elementary transpositions in $\Sym_r$.

In the theorem below we will consider the representation of ${\mf B}_{\sleq s}(n)$ on $T^{\sleq s}(V)$ as given by the homomorphism $E$ defined above. First we make a preliminary observation. Let $B\in\Hom_k(V\otimes V,k)\subseteq\End_k(T^{\sleq s}(V)$ be the endomorphism that maps $x_1\otimes x_2$ to $\la x_1,x_2\ra$ and let $\ov{B}\in\Hom_k(k,V\otimes V)\subseteq\End_k(T^{\sleq s}(V)$ be the endomorphism that maps $1$ to $\sum_{i=1}^n\e_ie_i\otimes e_{i'}$. Then we have
\begin{equation}\label{eq.B}
E(b_r)=B\otimes\id_{\otimes^rV}\text{\quad and\quad}E(\ov{b}_r)=\ov{B}\otimes\id_{\otimes^rV}.
\end{equation}
Note furthermore that $E(c_r)$ is just the projection of $T^{\sleq s}(V)$ onto $\otimes^rV$.

\begin{thmgl}\label{thm.doublecentraliser}
The following holds.
\begin{enumerate}[(i)]
\item $\End_{\Sp_n}(T^{\sleq s}(V))$ coincides with the image of ${\mf B}_{\sleq s}(n)$ in $\End_k(T^{\sleq s}(V))$.
\item $\End_{{\mf B}_{\sleq s}(n)}(T^{\sleq s}(V))$ is the enveloping algebra of $\Sp_n$ in $\End_k(T^{\sleq s}(V))$.
\item If $m\ge s$, then the homomorphism ${\mf B}_{\sleq s}(n)\to\End_k(T^{\sleq s}(V))$ is injective.
\item If $m<s$, then the homomorphism ${\mf B}_s(n)\to\End_k(\otimes^sV)$ is not injective.
\end{enumerate}
\end{thmgl}

\begin{proof}
(i).\ We follow Brauer's method as adjusted to our situation by Weyl, see \cite[V.2, p 141-142]{Wey}. There is a natural action of $\Sp_n$ on $\End_k(T^{\sleq s}(V))$ and it is clear that $\End_{\Sp_n}(T^{\sleq s}(V))$ consists of the $\Sp_n$-invariant elements of $\End_k(T^{\sleq s}(V))$. The vector spaces in \eqref{eq.moduleiso1} and \eqref{eq.moduleiso2} all have a natural $\Sp_n$-action and the isomorphisms there are $\Sp_n$-equivariant. It therefore suffices to show that for $u,v\in\{1,\ldots,s\}$, $\big((\otimes^uV^*)\otimes(\otimes^vV)\big)^{*\,\Sp_n}$ is spanned by the multilinear functions $F(D)$, where $D$ is a $(u,v)$-diagram and $F$ is given by \eqref{eq.F(D)}.

Let $u,v\in\{1,\ldots,s\}$ and put $w=u+v$. We will identify $V^*$ with $V$ by means of the isomorphism $\varphi:V^*\stackrel{\sim}{\to}V$. This means that the $y_i$ are now vector variables. We put $z_i=y_i$ for $i\in\{1,\ldots,u\}$ and $z_i=x_{i-u}$ for $i\in\{u+1,\ldots,w\}$. Since $k$ is infinite, $k[\oplus^wV]$ can be identified with the polynomial ring in the components of the $z_i$.
If $f\in k[\oplus^wV]$ is $\Sp_n$-invariant, then it is also $\Sp_n(K)$-invariant as an element of $K\otimes_kk[\oplus^wV]$, since $\Sp_n$ is dense in $\Sp_n(K)$. But then $f$ is a formal invariant in the definition of \cite[\S 2]{DeCProc}, see e.g. \cite[Remark I.2.8]{Jan}. We can now apply the first fundamental theorem of invariant theory for the symplectic group, \cite[Thm.~6.6]{DeCProc}. This gives us that $k[\oplus^wV]^{\Sp_n}$ is generated as a $k$-algebra by the scalar products $\la z_i,z_j\ra$, $1\le i<j\le w$. It follows immediately that $(\otimes^wV)^{*\,\Sp_n}$ is spanned by the monomials in the $\la z_i,z_j\ra$, $1\le i<j\le w$, in which each $z_i$ occurs exactly once. These monomials are precisely the $F(D)$, where $D$ is a $(u,v)$-diagram.\\
(ii).\ We have $k\la\mf S\ra\subseteq{\mf B}_{\sleq s}(n)$, where $\mf S$ is as in Proposition~\ref{prop.envalg}(iv). So, as in the proof of that result, we have that $\End_{{\mf B}_{\sleq s}(n)}(T^{\sleq s}(V))$ consists of $(s+1)$-tuples of bisymmetric substitutions. Using \eqref{eq.B} one easily checks that the condition of commuting with $E(b_r)$ is given by the equation~\eqref{eq.envalgeqs1} and  that \eqref{eq.envalgeqs2} gives the condition of commuting with $E(\ov{b}_r)$. By Theorem~\ref{thm.symplecticideal}(iii) these are precisely the equations that define ${\mc E}^{\sleq s}(\Sp_n)$.

The arguments in the proofs of (iii) and (iv) below are very similar to those for the orthogonal group in \cite[p 149]{Wey}. Recall that the {\it Pfaffian} of an alternating $2r\times 2r$-matrix A is defined by
\begin{equation}\label{eq.pfaff}
{\rm Pf}(A)=\sum_\pi\sgn(\pi)\prod_{i=1}^ra_{\pi_{2i-1},\pi_{2i}}\ ,
\end{equation}
where the sum is over all permutations $\pi$ of $\{1,2,\ldots,2r\}$ with $\pi_1<\pi_3<\cdots<\pi_{2r-1}$ and $\pi_{2i-1}<\pi_{2i}$ for all $i\in\{1,\ldots,r\}$.\\
(iii).\ By the isomorphisms \eqref{eq.moduleiso1} and \eqref{eq.moduleiso2} it is sufficient to show that for each $u,v\in{1,\ldots,s}$ with $u\equiv v\ ({\rm mod}\ 2)$, the multilinear functions $F(D)$, $D$ a $(u,v)$-diagram, are linearly independent. We use the identification $V^*\cong V$ and the notation of (i). As we have seen in (i) the $F(D)$ are the monomials in the $\la z_i,z_j\ra$, $1\le i<j\le w$ in which each $z_i$ occurs exactly once. The second fundamental theorem for the symplectic group, \cite[Thm.~6.7]{DeCProc}, says that the ideal of relations between the $\la z_i,z_j\ra$, $1\le i<j\le w$ is generated by the Pfaffians of the principal submatrices of size $n+1$ of the $w\times w$ alternating matrix $\la z_i,z_j\ra_{1\le i,j\le w}$.\footnote{A principal submatrix is obtained by choosing rows and columns from the same index set.} Since, by assumption, $w\le 2s\le n$, there are no such submatrices, that is, the $\la z_i,z_j\ra$, $1\le i<j\le w$ are algebraically independent. This means that the monomials in the $\la z_i,z_j\ra$, $1\le i<j\le w$ are linearly independent.\\
(iv).\ From \eqref{eq.pfaff} it is clear that ${\rm Pf}(A)$ is a signed sum of distinct monomials in the $a_{ij}$, $1\le i<j\le 2r$, in which each number $i\in\{1,2,\ldots,2r\}$ occurs precisely once either as row or as column index. Since $n<2s$, we must have that the alternating $2s$-multilinear map ${\rm Pf}(\la z_i,z_j\ra_{1\le i,j\le 2s})$ on $\oplus^{2s}V$ is zero. This gives us a nontrivial dependence relation between the multilinear functions $F(D)$, $D$ an $(s,s)$-diagram.
\end{proof}

\begin{cornn}
The algebra ${\mf B}_{\sleq s}(t)$ is associative for any $t\in k$.
\end{cornn}

\begin{proof}
First we observe that we can define the Brauer algebra over any commutative ring $R$ and then we have
${\mf B}_{\sleq s}(R,t_0)\cong R\otimes_{\mb Z[t]}{\mf B}_{\sleq s}(\mb Z[t],t)$ for all $t_0\in R$, where we now consider $t$ as an indeterminate and the homomorphism $\mb Z[t]\to R$ is given by $t\mapsto t_0$. So it suffices to show that ${\mf B}_{\sleq s}(\mb Z[t],t)$ is associative. If $D,D',D''\in{\mf B}_{\sleq s}(\mb Z[t],t)$, then the coefficients of $(DD')D''-D(D'D'')$ with respect to the diagram basis are polynomials in $t$. By Theorem~\ref{thm.doublecentraliser}(iii) applied with $k=\mb Q$, these polynomials vanish whenever we specialise $t$ to an even integer $\ge 2s$. It follows that these polynomials are identically zero.
\end{proof}

\begin{remsgl}
1.\ One can also derive the double centraliser theorem for $\Sp_n$ acting on $\otimes^sV$ as otained in \cite[Prop.~1.3 and Thm~1.4]{Dotyetal} by our method. Simply combine Remark~\ref{rems.defeqs}.2 and the arguments of the proof of Theorem~\ref{thm.doublecentraliser}.\\
2.\ Drop the assumption that $n$ is even and assume that the characteristic of $k$ is not $2$. Orthogonal versions ${\mf A}_{\sleq s}(t)$ and ${\mf A}_s(t)$ of ${\mf B}_{\sleq s}(t)$ and ${\mf B}_s(t)$ can also be defined. One just has to omit the sign in the definition of the multiplication. The algebra ${\mf A}_s(t)$ coincides with the orthogonal Brauer algebra defined in \cite{HanWal}. Furthermore the algebas ${\mf A}_{\sleq s}(n)$ and ${\mf A}_s(n)$ have natural representations in $T^{\sleq s}(V)$ and $\otimes^s V$.

Let $I$ be the orthogonal ideal. In characteristic zero homogeneous generators for $I_{hom}$ are given in \cite[Thm.~10.5a)]{Doty} and generators for $I$ with property \eqref{eq.idealgensprop} are given in \cite[Thm.~5.2.C]{Wey}. The Corollary to Proposition~\ref{prop.envalg} and the proof of Theorem~\ref{thm.doublecentraliser} show that, to prove the double centraliser theorem over $k$ for the orthogonal group acting on $\otimes^s V$ and $T^{\sleq s}(V)$, it is enough to show that these elements are also generators of $I_{hom}$ and generators for $I$ with property \eqref{eq.idealgensprop} over $k$.\\
3.\ Assume that $k=\mb C$. Then it can be shown by easy invariant theoretic arguments as in the proof of Theorem~\ref{thm.doublecentraliser}(iii) and (iv) that the natural representation ${\mf A}_s(n)\to\End_k(\otimes^sV)$ is injective if $n\ge 2s$ and not injective if $n<s$; see \cite[p 149]{Wey}. So, contrary to the symplectic case, these arguments do not give a complete criterion for faithfulness. In \cite{Bro} it is proved that the natural representation of ${\mf A}_s(n)$ is faithful if and only if $n\ge s$.\\
4.\ From \cite[Thm.'s 2.11, 4.6, 4.3]{MoWa} it follows that ${\mf A}_s(t)$ has the well-known presentation given in e.g. \cite[p 3]{Dotyetal}. From this one easily deduces a presentation for ${\mf B}_s(t)$ and an isomorphism ${\mf B}_s(t)\cong {\mf A}_s(-t)$. This is in accordance with the results in \cite{HanWal} and \cite{Wen}.
\end{remsgl}

\noindent{\it Acknowledgement}. I would like to thank S.\ Oehms for his comments on a preliminary version of this paper and A.\ M.\ Cohen for bringing \cite{MoWa} to my attention. This research was funded by the EPSRC Grant EP/C542150/1.

\end{document}